\documentclass[11pt]{amsart}
\usepackage{amsmath,amsfonts,amsthm,amssymb,amscd,url}
\def\classification#1{\def\@class{#1}}
\classification{\null}

\DeclareFontFamily{OT1}{rsfs}{}
\DeclareFontShape{OT1}{rsfs}{n}{it}{<-> rsfs10}{}
\DeclareMathAlphabet{\mathscr}{OT1}{rsfs}{n}{it}

\DeclareMathOperator{\mo}{\,mod}

\DeclareMathOperator{\Sym}{Sym}
\DeclareMathOperator{\Alt}{Alt}
\DeclareMathOperator{\Disc}{Disc}

\DeclareMathOperator{\Frob}{Frob}

\DeclareMathOperator{\Gal}{Gal}

\DeclareMathOperator{\Cl}{Cl}
\DeclareMathOperator{\disc}{Disc}

\newtheorem{prop}{Proposition}[section]

\newtheorem*{main}{Main Theorem}

\newtheorem{lem}[prop]{Lemma}

\numberwithin{equation}{section}
\title{Square-free values of $f(p)$, $f$ cubic}
\author{Harald Andr\'es Helfgott}
\address{H. A. Helfgott, \'Ecole Normale Sup\'erieure, D\'epartement de 
Math\'ematiques, 45 rue d'Ulm, F-75230 Paris, France}
\email{harald.helfgott@ens.fr}
\begin{document}
\begin{abstract}
Let $f\in \mathbb{Z}\lbrack x\rbrack$, $\deg(f)=3$. Assume that $f$ does not have repeated roots. Assume as well that, for every prime $q$, 
$f(x) \not\equiv 0 \mo q^2$ has at least one solution in 
$(\mathbb{Z}/q^2 \mathbb{Z})^*$. Then, under these two necessary conditions, there are infinitely many
primes $p$ such that $f(p)$ is square-free.
\end{abstract}
\maketitle

\section{Introduction}
An integer is said to be {\em square-free} if it is not divisible by the square $d^2$ of any integer $d$ greater than $1$. It is easy to prove that for 
$f(x) = m x + a$, $a,m\in \mathbb{Z}$, there are infinitely many integers $n$
such that $m n + a$ is square-free -- provided, of course, that 
$\gcd(a,m)$ is square-free.

For $f$ quadratic, the infinity of integers $n$ such that $f(n)$ is square-free
was proved by Estermann \cite{Es} in 1931. (Again, there are necessary
conditions that have to be fulfilled: $f$ should not have repeated roots
(i.e., for $\deg(f)=2$, $f$ should not be a constant times a square) and
$f(x)\not\equiv 0 \mo q^2$ should have a solution in $\mathbb{Z}/q^2 \mathbb{Z}$
for every prime $q$.) 

For $f$ cubic, the fact that there are infinitely many
 integers $n$ such that $f(n)$ is square-free was proven by Erd\H{o}s \cite{Er}. (See also \cite[Ch.\ IV]{Hoo}.) It can
be argued that Erd\H{o}s's proof wittily avoids several underlying issues, some of
which are  diophantine problems that are far from trivial. Perhaps because
of this, Erd\H{o}s posed the problem of proving that $f(p)$ is square-free
for infinitely many {\em primes} $p$. The diophantine issues then become
unavoidable, and the problem becomes much harder. 

The paper \cite{Hbr} settled the issue for $f$ cubic with Galois group 
$\Alt(3)$. Unfortunately, most cubics have Galois group $\Sym(3)$.

The present paper solves the problem for all $f$ cubic.
\begin{main}
Let $f\in \mathbb{Z}\lbrack x\rbrack$ be a cubic polynomial without repeated
roots. 
Then the number of prime numbers $p\leq N$ such that $f(p)$ is square-free is
\begin{equation}\label{eq:expr}
 (1+ o_f(1)) \prod_{\text{$q$ prime}} \left(1 - \frac{\rho_f(q^2)}{\phi(q^2)}\right) \frac{N}{\log N} + O(1),
 \end{equation}
where $\rho_f(q^2)$ is the number of solutions to $f(x)\equiv 0 \mo q^2$ in 
$(\mathbb{Z}/q^2 \mathbb{Z})^*$.
\end{main}
It is easy to show that, if $f(x)\not\equiv 0 \mo q^2$ has at least one solution
in $(\mathbb{Z}/q^2 \mathbb{Z})^*$ for every prime $q$ smaller than a constant depending only on $f$, then the infinite product in (\ref{eq:expr}) converges to a non-zero value (see the remark 
at the end of \S \ref{sec:oldhat}). In other words, we have a necessary and
sufficient condition for the product in (\ref{eq:expr}) to be non-zero,
and this condition is such that it can be checked explicitly in time
$O_f(1)$. 

The analogous problem -- namely, proving that, for a polynomial $f$ of degree $k$
satisfying the necessary conditions as above, there is an infinite number of 
primes $p$ such that $f(p)$ has no divisors of the form $d^{k-1}$, $d>1$ --
was solved by Nair \cite{Na} for $k\geq 7$. Several cases with $k=3,4,5,6$ were
solved in \cite{Hbr}; see the list in \cite[(1.3)]{Hbr}. A summary of 
the proof in this paper appeared previously in  \cite{Hsh}. Since then,
the cases of $k=5,6$ have been settled by Browning \cite[Thm.\ 2]{Br},
building in part on arguments by Salberger
\cite{Sa} and Heath-Brown \cite{HeBr}. As a consequence,
only the case of polynomials $f$ of degree $k=4$ with Galois group $\Alt(4)$
or $\Sym(4)$ remains open. 

The author's interest in the problem was first sparked by his work on root
numbers of elliptic curves. There are indeed many problems in number theory
where matters become much simpler technically if one assumes one is working
with square-free numbers.
This is the natural domain of application of the results in this paper.
\subsection{Notation}
In this paper, $p$ and $q$ always denote primes. We write $\omega(d)$ for the number
of prime divisors of an integer $d$, and $\tau_k(d)$ for the number of tuples
of positive integers $(m_1,\dotsc,m_k)$ such that $d = m_1 m_2 \dotsb m_k$.
Given a prime $p$ and a non-zero integer $n$, the valuation $v_p(n)$ is the 
largest non-negative integer $r$ such that $p^r|n$. 
Given positive integers $n$ and $m$, we write
$\gcd(n,m^\infty)$ for $\prod_{p|m} p^{v_p(n)}$. Let $\pi(N)$ be the number
of primes $\leq N$.

Let $K$ be a number field with Galois group $\Gal(K/\mathbb{Q})$.
We write $\omega_K(d)$ for the number of prime ideals
dividing $d$ in a number field $K$. Given a rational prime $p$ unramified in
$K/\mathbb{Q}$, we denote by $\Frob_p\subset \Gal(K/\mathbb{Q})$ the Frobenius
symbol of $p$; it is always a conjugacy class in $\Gal(K/\mathbb{Q})$.
For $g\in \Gal_f$, we write
$\omega_{\Cl(g)}(n)$ for the number of prime divisors $p|n$
such that $\Frob_p = \Cl(g)$, where 
$\Cl(g)$ is the conjugacy class of $g$.
\subsection{Acknowledgements}

Thanks are due to M. Dimitrov, G. Harcos and M. Hindry for answering my 
questions regarding a possible conditional generalisation of the present paper to the case of polynomials of higher degree, and to S. Ganguly and M. Hindry
for very useful discussions.

The results in this paper were largely
 proven at the Universit\'e de Montr\'eal towards
the end of the author's stay as a CRM-ISM fellow. The paper itself
was written in part during a stay at EPFL, Lausanne,
Switzerland.  The author is thankful to both 
A. Granville and Ph. Michel
for having provided good working environments. 

\section{Reduction to the problem of large square factors $q^2|f(x)$, $q$ prime}\label{sec:oldhat}

We wish to reduce the problem of estimating the number of primes 
$p\leq N$ such that $f(p)$ is square-free to the problem of bounding from above the number of primes $p\leq N$ such that $f(p)$ has a square factor of the form $q^2$, $q$ prime, $q>
N (\log N)^{-\epsilon}$. If we cared about minimising the error term, this would be a non-trivial
problem; see the treatment in \cite[\S 3]{Hsq}. As it happens, the error terms we will get later from other sources will be fairly large anyhow, and thus we can afford to carry out things in this section in a way that is easy and classical.
(See \cite[Ch. IV]{Hoo} or \cite{Gr}, for instance.)

In what follows, $p$ and $q$ always range over the primes. We have
\[\begin{aligned} 
|\{p\leq N: \text{$f(p)$ is square-free}\}| &=
|\{p\leq N:  q^2|f(p) \Rightarrow q \geq \frac{1}{3} \log N\}| \\ &+
O(|\{p\leq N:  \exists q  \;\text{s.t.}\; q^2|f(p),\; q\geq \frac{1}{3} \log N  \}|).
\end{aligned}\]

By the inclusion-exclusion principle and Bombieri-Vinogradov,
 \[\begin{aligned} 
 |\{p\leq N:  q^2|f(p) \Rightarrow q \geq \frac{1}{3} \log N\}| &=
 \mathop{\sum_{d}}_{q|d \Rightarrow q < \frac{1}{3} \log N} \mu(d) 
|\{p\leq N:  d^2|f(p)\}| \\
&=
 \mathop{\sum_{d}}_{q|d \Rightarrow q < \frac{1}{3} \log N} \mu(d) 
 \rho_f(d^2) \frac{\pi(N)}{\phi(d^2)} + O(N/(\log N)^{100})\\ &= 
\mathop{\prod_{\text{$q$ prime}}}_{q< \frac{1}{3} \log N} 
\left(1 - \frac{\rho_f(q^2)}{\phi(q^2)}\right) \pi(N)
 + O(N/(\log N)^{100})\\
 &=
\prod_{\text{$q$ prime}} 
\left(1 - \frac{\rho_f(q^2)}{\phi(q^2)}\right) \pi(N) + O(N/(\log N)^2).
 \end{aligned}\]
Recall as well that $\pi(N) = \frac{N}{\log N} + O\left(\frac{N}{(\log
    N)^2}\right)$ (Prime Number Theorem). 

At the same time, 
\[\begin{aligned} 
|\{p\leq N:  \exists q  \;\text{s.t.}\; q^2|f(p),\; q\geq \frac{1}{3} \log N  \}|
&\leq
|\{p\leq N:  \exists q  \;\text{s.t.}\; q^2|f(p),\; \frac{1}{3} \log N \leq q< N^{1/3}  \}|\\   &+
|\{p\leq N:  \exists q  \;\text{s.t.}\; q^2|f(p),\; N^{1/3} \leq q< N (\log N)^{-\epsilon}  \}|\\ &+
|\{p\leq N:  \exists q  \;\text{s.t.}\; q^2|f(p),\; q\geq  N (\log N)^{-\epsilon}  \}|\\
&\leq
\sum_{\frac{1}{3} \log N \leq q< N^{1/3}} O\left(\frac{N/(\log N)}{q (q-1)}\right) +
 O(N/(\log N)^{100})
\\ &+  \sum_{N^{1/3} \leq q< N (\log N)^{-\epsilon}} O\left(\frac{N}{q^2} + 1\right)\\
 &+|\{p\leq N:  \exists q  \;\text{s.t.}\; q^2|f(p),\; q\geq  N (\log
 N)^{-\epsilon}  \}|,\end{aligned}\]
where we have used Brun-Titchmarsh (or any upper-bound sieve)
to justify the second inequality, and where, as per our convention,
$q$ ranges only over the primes. 
The series on the right side sum up to 
$O(N/(\log N)^2)$ and $O(N/(\log N)^{1+\epsilon})$, respectively; hence
\[\begin{aligned} 
|\{p\leq N:  \exists q  \;\text{s.t.}\; q^2|f(p),\; q\geq \frac{1}{3} \log N  \}|
&\leq
|\{p\leq N:  \exists q  \;\text{s.t.}\; q^2|f(p),\; q\geq  N (\log
 N)^{-\epsilon}  \}|\\ &+
O\left(\frac{N}{(\log N)^{1+\epsilon}}\right).
\end{aligned}\]
Therefore
\begin{equation}\label{eq:ali}\begin{aligned}
|\{p\leq N: \text{$f(p)$ is square-free}\}| &= \prod_{\text{$q$ prime}} 
\left(1 - \frac{\rho_f(q^2)}{\phi(q^2)}\right)\cdot \frac{N}{\log N} 
+
O\left(\frac{N}{(\log N)^{1+\epsilon}}\right)
\\ &+
 |\{p\leq N:  \exists q  \;\text{s.t.}\; q^2|f(p),\; q\geq  N (\log N)^{-\epsilon}  \}| 
 \end{aligned}\end{equation}
for any $\epsilon>0$.

The only thing that remains is to bound
$ |\{p\leq N:  \exists q  \;\text{s.t.}\; q^2|f(p),\; q\geq  N (\log N)^{-\epsilon}  \}| $.
This problem will occupy us in the rest of the paper.

In the meantime, let us note that $\rho_f(q^2) \leq \deg(f)$ for every $q$
larger than a constant depending only on $f$ (by Hensel's lemma). Hence
the infinite product in (\ref{eq:ali}) is non-zero provided that 
$\rho_f(q^2) < \phi(q^2)$ (i.e., provided that $f(x)\not\equiv 0 \mo q^2$
has at least one solution in $(\mathbb{Z}/q^2 \mathbb{Z})^*$) for every
$q$ smaller than a constant depending only on $f$. If there is a $q$
such that $f(x)\not\equiv 0 \mo q^2$ has no
solutions in $(\mathbb{Z}/q^2 \mathbb{Z})^*$, then $f(p)$ can be
square-free only when $\gcd(p,q^2)\ne 1$; obviously, 
$\gcd(p,q^2)\ne 1$ can happen for at
most one value of $p$, namely, $p=q$. (This is where the term $O(1)$ in 
(\ref{eq:expr}) comes from.)

\section{Integer points on a typical quadratic twist of an elliptic curve}

Consider two points $(x_1,y_1)$, $(x_2,y_2)$ ($x_i, y_i\in \mathbb{Z}$) on
the curve $d y^2 = f(x)$. This is an elliptic curve. It is well known that points with integer
coordinates on an elliptic curve tend to repel each other; this was already used in the present context in
\cite{Hsq} (see also the earlier work \cite{Si}). As was pointed out in \cite{HV}, two points
repel each other more strongly if their coordinates are congruent to each other modulo some
large integer. (This is a somewhat intuitive description; we will do things
rigorously below.)

In \cite{Hbr}, I used this phenomenon on the curve $d y^2 = f(x)$. I first
showed by elementary means
that most integers $d\leq N$ have large factors $d_0|d$, $d_0>N^{1-\epsilon}$, such that $d_0$ has
few prime divisors. It is then the case that the $x$-coordinates of the points $(x,y)$ on the
curve fall into few congruence classes modulo $d_0$ (because $d_0$ has few
prime divisors). Moreover, by the argument on elliptic curves just given,
 there can be only
few points whose $x$-coordinates are in a given congruence class modulo $d_0$
(because $d_0$ is large, and makes points in such a congruence class repel each other strongly). It follows that there are few points $(x,y)$
($x,y \in \mathbb{Z}$, $1\leq x,y\leq N$) on the curve $d y^2 = f(x)$, unless $d$ is in some small
exceptional set.

We carry out this argument again, largely just by citing \cite{Hsq} and \cite{Hbr}.

\begin{prop}\label{prop:smir}
Let $f\in \mathbb{Z}\lbrack x \rbrack$ be a polynomial of degree $3$ with no repeated roots. Let 
$d$ be a square-free integer. Then, for any $N$,  the number of integer solutions 
$(x,y)\in \mathbb{Z}^2$ to $d y^2 = f(x)$ with $N^{1/2} <x \leq N$ is at most
\begin{equation}\label{eq:onestar}
O_f\left(C^{\omega(d)}\right),
\end{equation}
where $C$ is an absolute constant.
\end{prop}
This bound is an immediate consequence of \cite{Si}, Theorem A, which is already based
on the idea of repelling points (and does not require the condition $N^{1/2} <x \leq N$). The alternative proof in \cite[Cor.\ 4.18]{Hsq}
provides an explicit value for $C$ by means of sphere-packing bounds
\cite{KL}.
\begin{proof}
By \cite[Cor.\ 4.18]{Hsq} (applied with $\epsilon=1/2$)
and any rank bound obtained by descent, e.g.,
 the standard bound in \cite[Prop.\ 7.1]{BK} (namely,
 $\text{rank} \leq \omega_K(d) - \omega(d) + O_f(1)
 \leq 2 \omega(d) + O_f(1)$, 
where $K=\mathbb{Q}(\alpha)$ and $\alpha$ is a root of $f(\alpha)=0$.)
\end{proof}

\begin{prop}\label{prop:steroid}
Let $f\in \mathbb{Z}\lbrack x\rbrack$ be a polynomial of degree $3$ with no repeated roots.
Let $d\leq X$ be a positive integer. Suppose that $d$ has an integer divisor 
$d_0\geq X^{1-\epsilon}$, $\epsilon>0$. Assume furthermore that
$\gcd(d_0, 2 \disc f)=1$. Then the
number of integer solutions $(x,y)\in \mathbb{Z}^2$ 
to $d y^2 =f(x)$ with $X^{1-\epsilon} < x \leq X$ is at most
\begin{equation}\label{eq:twostar}
O_{f,\epsilon}\left(e^{O_f(\epsilon \omega(d))} 3^{\omega(d_0)}\right).
\end{equation}
\end{prop}
This bound uses the divisor $d_0$ in order to increase repulsion in the way outlined above.
If a $d_0$ with few prime divisors is chosen, the bound (\ref{eq:twostar}) will be much smaller
than (\ref{eq:onestar}).
\begin{proof}
This is a special case ($\deg(f)=3$, $k=2$, $c=2$) of \cite[Prop.\ 4.3]{Hbr}.
\end{proof}

We now need two lemmas on the integers.
\begin{lem}\label{lem:tofuhunter}
Let $f\in \mathbb{Z}\lbrack x\rbrack$ be a polynomial.
For any $A>0$ and for all but
$O_A(N (\log N)^{-A})$ integers $n$ between $1$ and $N$, the number of prime
divisors $\omega(f(n))$ of $f(n)$ is $O_{A,f}(\log \log N)$.
\end{lem}
\begin{proof}
This is standard. If $f(n)$ has $\geq C \log \log N$ prime factors, then it has
$\geq \frac{C}{\deg(f)} \log \log N$ prime factors (namely, the 
$\frac{C}{\deg(f)} \log \log N$ smallest ones) whose
product is $\ll_f N$. Their products give us 
$\geq 2^{\frac{C}{\deg(f)} \log \log N} = (\log N)^{C (\log 2)/\deg(f)}$ divisors
$d\ll_f N$ of $f(n)$. At the same time,
\[\begin{aligned}
\sum_{n\leq N} \sum_{d|f(n), d\leq N} 1&= \sum_{d\leq N} \mathop{\sum_{n\leq N}}_{d|f(n)} 1\\
&\leq \sum_{d\leq N} \left(\frac{N}{d} + 1 \right) \cdot (\deg(f))^{\omega(d)}
 \ll N (\log N)^{B},
\end{aligned}\]
where $B=O_f(1)$. Thus, there can be at most $N (\log N)^{-(C (\log 2)/\deg(f) - B)}$ integers
 $n\leq N$ such that $f(n)$ has $\geq C \log \log N$ prime factors. We set $C$ so that
 $\frac{C \log 2}{\deg(f)} - B \geq A$ and we are done.
\end{proof}

\begin{lem}\label{lem:seitan}
Let $f\in \mathbb{Z}$ be a polynomial.
For any $A>0$, $\epsilon>0$, $m>0$, it is the case that, for all but
$O_{A,\epsilon,m}(N (\log N)^{-A})$ integers $n$ between $1$ and $N$,
there is a divisor $d_1|f(n)$ such that $d_1<N^{\epsilon/2}$,
$\omega(f(n)/d_1)<\epsilon \log \log X$, and $\gcd(f(n)/d_1,m)=1$.
\end{lem}
\begin{proof}
Let $\delta(N)$ be as in \cite[Lem.\ 5.2]{Hbr} with $\epsilon/4$ instead of $\epsilon$. (That is, we let $\delta(N) = (\log N)^{-\epsilon/4 r e^{2 r}}$, where
$r=\deg(f)$.)
Let \[d_1 = \gcd(f(n),m^{\infty})\cdot 
\prod_{p|f(n),p\nmid m: p\leq N^{\delta(N)}} p .\] By definition,
$\gcd(f(n)/d_1,m)=1$. Also, by
\cite[Lem.\ 5.2]{Hbr}, we know that $\omega(f(n)/d_1)< \epsilon \log \log N$
and $\prod_{p|n: p\leq N^{\delta(N)}} p < N^{\epsilon/4}$ for all but
$O_{A,\epsilon}(N (\log N)^{-A})$ integers $n$ between $1$ and $N$.

Now
\[\begin{aligned}
\sum_{n\leq N} \gcd(n,m^\infty) &\leq \sum_{d|m^\infty} \mathop{\sum_{n\leq N}}_{d|n} d\\
&\leq \mathop{\sum_{d|m^\infty}}_{d\leq N} N 
\leq N\cdot \prod_{p|m} \mathop{\sum_{\alpha \geq 1}}_{p^\alpha \leq N} 1
\ll N (\log N)^{\omega(m)} \ll_{m,\epsilon} N^{1 + \epsilon/8}.
\end{aligned}\]
It follows that, for all but $O_{m, \epsilon}(N^{1-\epsilon/8} )$ integers $n$ between $1$ and $N$,
$\gcd(f(n),m^{\infty}) \leq N^{\epsilon/4}$.
 Hence $d_1\leq N^{\epsilon/2}$.
\end{proof}

\begin{prop}\label{prop:spire}
Let $f\in \mathbb{Z}\lbrack x\rbrack$
 be a polynomial of degree $3$ with no repeated roots.
Let $D$ be a set of positive integers.
Then the total number of integers $x$ with $1\leq x\leq N$ such that
\[d y^2 =f(x)\]
for some integer $y\geq N (\log N)^{-\epsilon}$ and some $d\in D$ is at most
\begin{equation}\label{eq:mundi}O_{f,\varepsilon}\left( |D| (\log N)^{\varepsilon}\right) +
O_{f,A,\varepsilon}\left(N (\log N)^{-A}\right)\end{equation}
for arbitrary $A$ and $\varepsilon>0$.
\end{prop}
\begin{proof}
Let $\epsilon>0$ be a small parameter to be set later.
If $d y^2 = f(x)$ for some integer $y$ and some integer
 $d<N^{1-\epsilon/4}$, then $d' (y')^2 = f(x)$ for some 
integer $y'$ and some square-free integer $d'<N^{1-\epsilon/4}$. By 
Prop.\ \ref{prop:smir} and Lemma \ref{lem:tofuhunter},
the total number of $x\leq N$ satisfying such an equation is 
$(\log N)^{O_{f,A}(1)} N^{1-\epsilon/4} +O_A(N (\log N)^{-A}) \ll_{A,f,\epsilon}
N (\log N)^{-A}$ for $A$ arbitrarily large. 

Let, then, $d y^2 =f(x)$, $d\geq N^{1-\epsilon/4}$, $x\leq N$,
$y\geq N (\log N)^{-\epsilon}$.
By Lemma \ref{lem:tofuhunter}, we can assume 
that $\omega(d) \ll_{A,f} \log \log N$
(taking out at most $O_A\left(N (\log N)^{-A}\right)$ values of $x$).
By Lemma \ref{lem:seitan}, we can assume (taking out at most
$O_{A,f,\epsilon}(N (\log N)^{-A})$ values of $x$)
that there is a $d_1|f(x)$ such that
 $d_1<N^{\epsilon/2}$, $\omega(f(n)/d_1) < \epsilon \log \log N$ and
 $\gcd(f(n)/d_1, 2 \disc(f))=1$. Let $d_0 = d/\gcd(d,d_1)$. Then 
 $d_0\geq d/N^{\epsilon/2} > N^{1-3 \epsilon/4}$,
 $\omega(d_0) < \epsilon \log \log N$, $\omega(d) \ll_{A,f} \log \log N$ and
 $\gcd(d_0,2\disc(f))=1$.
 
Since $y\geq N (\log N)^{-\epsilon}$ and $d = f(x)/y^2$, we have $d\leq C_f N
(\log N)^{2 \epsilon}$ for some constant $C_f$. 
We apply Prop.\  \ref{prop:steroid} with $X = C_f N (\log N)^{2 \epsilon}$.
(The condition $d_0\geq X^{1-\epsilon}$ is fulfilled by
$d_0>N^{1-3\epsilon/4}$
provided $N$ is larger than a constant $c_{f,\epsilon}$
depending only on $f$ and $\epsilon$;
we can assume $N$ is larger than $c_{f,\epsilon}$ because conclusion
(\ref{eq:mundi}) is otherwise trivial.)
and obtain that the number of integer solutions to $d y^2 = f(x)$ is at most
\[ \ll_{f,\epsilon} |D| e^{O_{f,A}(\epsilon) \log \log N} 3^{\epsilon \log \log N} 
\]
(taking out at most $\ll_{A,f,\epsilon}  N (\log N)^{-A} + X^{1-\epsilon} \ll_{A,f,\epsilon}
N (\log N)^{-A}$ values of $x$). 
For $\epsilon$ small enough in terms of $f$, $A$ and $\varepsilon$,
this is $\leq |D| (\log N)^\varepsilon$, as desired.
\end{proof}

In view of (\ref{eq:mundi}), what remains is to show that we can eliminate most possible
values of $d$ in $d q^2 = f(p)$, $p\leq N$, $q\geq N (\log N)^{-\epsilon}$,
where we allow ourselves to take out first a proportion $o(1)$ of all possible values of $p\leq N$.

\section{Typical properties of $f(q)$ and $d=f(p)/y^2$}

Les $\alpha$ be a root of $f(\alpha)=0$. Let 
$\Gal_f = \Gal(\mathbb{Q}(\alpha)/\mathbb{Q})$. For $g\in \Gal_f$, let
$\omega_{\Cl(g)}(n)$ be the number of prime divisors $p|n$ unramified
in $\mathbb{Q}(\alpha)/\mathbb{Q}$
such that $\Frob_p = \Cl(g)$.
Let $\alpha_{\Cl(g)}$ be the number of fixed points of any 
representative $g$ of $\Cl(g)$, considered
as a permutation map on the roots of $f(x)=0$ in $\mathbb{C}$.
It is a standard fact that $\alpha_{\Cl(g)}$ equals the number of
roots $x\in \mathbb{Z}/p\mathbb{Z}$ of $f(x)\equiv 0 \mo p$
for any $p$ unramified in $\mathbb{Q}(\alpha)/\mathbb{Q}$
 such that $\Frob_p = \Cl(g)$. 

As is usual, we write the number of points on the curve $y^2=f(x) \mo p$
as $p+1-a_p$, where $a_p$ is an integer.

Our aim in this section is to show that, for a proportion $1+o(1)$ of all
primes $q\leq N$,
\begin{enumerate}
\item\label{it:qna}
 $\omega_{\Cl(g)}(f(q)) = (\alpha_{\Cl(g)} + o(1)) \frac{|\Cl(g)|}{\Gal_f}
\log \log N$ for every $g\in \Gal_f$, and
\item\label{it:qnb} $\sum_{p\leq z} \frac{a_p}{p} \left(\frac{f(q)}{p}\right) =
(1 + o(1)) \sum_{p\leq z} \frac{-a_p^2}{p^2}$ for 
$\frac{1}{o(1)} \leq z\leq N^{\delta}$, $\delta>0$ smaller than a constant.
\end{enumerate}

Here (\ref{it:qna}) is unsurprising; it is clear that the probability of
$p|f(q)$ for $p$ fixed and $q$ prime and random is $\frac{\alpha_f}{p}$,
and the sum $\sum_{p\leq N: \Frob_p = \Cl(g)} \frac{1}{p}$ is
$(1+ o(1)) \frac{|\Cl(g)|}{|\Gal_f|} \log \log N$ by Chebotarev's density 
theorem.

As for statement (\ref{it:qnb}),
the number of points on $y^2 = f(x) \mo p$ is
\[p+1-a_p = \sum_{x\in \mathbb{Z}/p\mathbb{Z}} \left(1 +
\left(\frac{f(x)}{p}\right)\right) 
= p + \sum_{x\in \mathbb{Z}/p\mathbb{Z}} \left(\frac{f(x)}{p}\right).\]
(Here and throughout the paper, $\left(\frac{\cdot}{\cdot}\right)$ and
$(\cdot/\cdot)$ stand for the Jacobi symbol.)
Hence, the expected value of $\left(\frac{f(q)}{p}\right)$ for $p$ fixed
and
$q$ prime and random should be \[\frac{1}{p-1} \left(1 - a_p -
\left(\frac{f(0)}{p}\right)\right) = - \frac{a_p}{p} + \text{error term}.\]
Thus, the expected value of $\sum_{p\leq z} \frac{a_p}{p}
 \left(\frac{f(q)}{p}\right)$ should be about 
$\sum_{p\leq z} \frac{-a_p^2}{p^2}$.

As elsewhere in this paper, we will carry out our arguments as is customary
in analytic number theory, inspired by the probabilistic reasoning detailed
above. (Alternatively, one could start by proving probabilistic statements
and deduce statements in number theory from them, as in \S 5--6 in \cite{Hbr}.
That option generally takes more space and work.)

Part of the point in estimating $\omega_{\Cl(g)}(f(q))$ and $(f(q)/p)$ is
that neither quantity changes much when $f(p)$ is divided by the square of
a prime: if $d = f(q)/y^2$, $y$ a prime, then
\begin{equation}\label{eq:katmid}
\begin{aligned}
\omega_{\Cl(g)}(f(p)) - 1 &\leq \omega_{\Cl(g)}(d) \leq \omega_
{\Cl(g)}(f(p))\\
\left(\frac{d}{p}\right) &= \left(\frac{f(q)}{p}\right)\;\;\;\;\;\;\;\;\;
 \text{for $p\ne y$.}
\end{aligned}
\end{equation}

Thus, what follows will help us determine later what form any $d$
satisfying $d y^2 = f(q)$ must take, where $y$ can be any prime and $q$ can
be any prime $\leq N$ outside a set of density $0$.\\

\begin{center}
* * *
\end{center}

We will prove both (\ref{it:qna}) and (\ref{it:qnb}) 
using, in essence, bounds on variances and Chebyshev's inequality. 

\begin{lem}\label{lem:plusca}
Let $f\in \mathbb{Z}\lbrack x\rbrack$ be a polynomial irreducible over $\mathbb{Q}\lbrack x\rbrack$. Let $g\in \Gal_f$. Let $z= z(N)$ be such that $\lim_{N\to \infty}
z(N) = \infty$ and $z < N^{1/4-\epsilon}$, $\epsilon>0$.
Then
\[\mathop{\sum_{p\leq z,\; \text{$p$ unramified}}}_{\Frob_p = \Cl(g),\;
p|f(q)} 1 = (\alpha_{\Cl(g)} + o_f(1)) \frac{|\Cl(g)|}{|\Gal_f|}
\log \log z
\]
for a proportion $1+o_{f,\epsilon}(1)$ of all primes $q\leq N$.
\end{lem}
The proof will 
not be very different from Tur\'an's classical proof that the average
 number
of prime divisors of an integer $\leq N$ is $\sim \log \log N$.
\begin{proof}
In what follows,
our sums over $p$ range only over primes $p$ unramified in 
$\mathbb{Q}(\alpha)/\mathbb{Q}$, $\alpha$ a root of $f$,
 whereas our sums over $q$ range over
all primes. 
We will give a variance bound, i.e., we will show that
\begin{equation}\label{eq:surt}
V = \sum_{q\leq N} \left|\mathop{\mathop{\sum_{p\leq z}}_{\Frob_p =
\Cl(g)}}_{p|f(q)} 1 - R
\right|^2
\end{equation}
is small, where 
$R = \sum_{p\leq z:\; \Frob_p = \Cl(g)} \frac{\alpha_{\Cl(g)}}{p}$.

Expanding (\ref{eq:surt}), we get
\begin{equation}\label{eq:kilkeat}\begin{aligned}
V &= R^2 \pi(N) - 2 R \mathop{\sum_{p\leq z}}_{\Frob_p = \Cl(g)}\; 
\mathop{\sum_{q\leq N}}_{p|f(q)} 1 \\ &+  
\mathop{\mathop{\sum_{p_1\leq z}}_{\Frob_{p_1} = \Cl(g)}\; 
\mathop{\sum_{p_1\leq z}}_{\Frob_{p_2} = \Cl(g)}}_{p_1 \ne p_2}\;
\mathop{\sum_{q\leq N}}_{p_1 p_2 | f(q)} 1 +
\mathop{\sum_{p\leq z}}_{\Frob_p = \Cl(g)} 
\mathop{\sum_{q\leq N}}_{p|f(q)} 1 .
\end{aligned}\end{equation}
Now 
\[\begin{aligned}
\mathop{\sum_{p\leq z}}_{\Frob_p = \Cl(g)}\;
\mathop{\sum_{q\leq N}}_{p|f(q)} 1 &=
\mathop{\sum_{p\leq z}}_{\Frob_p = \Cl(g)} 
|\{x\in (\mathbb{Z}/p\mathbb{Z})^* : f(x) = 0 \mo p\}|\cdot 
\frac{\pi(N)}{\phi(p)}\\ 
&+ O\left(z+
\sum_{p\leq z} 
\mathop{\max_{a\mo p}}_{\gcd(a,p)=1} \left(|\{q\leq N: q\equiv a \mo p\}| - 
 \frac{\pi(N)}{\phi(p)} \right)\right).
\end{aligned}\]
By Bombieri-Vinogradov (as in \cite[Thm.\ 0]{BFI}),
\[
\sum_{m\leq N^{1/2 - \delta}} 
\mathop{\max_{a\mo m}}_{\gcd(a,m)= 1} 
\left(|\{q\leq N: q\equiv a \mo m\}| - 
 \frac{\pi(N)}{\phi(m)} \right) \ll_{A,\delta} \frac{N}{(\log N)^A} 
\]
for all $A$, $\delta>0$. We also have
$|\{x\in (\mathbb{Z}/p\mathbb{Z})^* : f(x) = 0 \}| = \alpha_{\Cl(g)}$
for all (unramified) $p$ with $\Frob_p = \Cl(g)$. Hence
\[\begin{aligned}
\mathop{\sum_{p\leq z}}_{\Frob_p = \Cl(g)}
\mathop{\sum_{q\leq N}}_{p|f(q)} 1 &= 
\pi(N) \mathop{\sum_{p\leq z}}_{\Frob_p = \Cl(g)}
\frac{\alpha_{\Cl(g)}}{p-1}
 + O_A(N (\log N)^{-A})\\ &=
\pi(N) \left( O(1) + \mathop{\sum_{p\leq z}}_{\Frob_p = \Cl(g)}
\frac{\alpha_{\Cl(g)}}{p} \right)
 = \pi(N) (R+ O(1)). 
\end{aligned}\]
Similarly, we have
\[\begin{aligned}
\mathop{\sum_{p_1\leq z}}_{\Frob_{p_1} = \Cl(g)}
\mathop{\sum_{p_2\leq z}}_{\Frob_{p_2} = \Cl(g)}
\mathop{\sum_{q\leq N}}_{p_1 p_2|f(q)} 1 &= \pi(N) 
\mathop{\sum_{p_1\leq z}}_{\Frob_{p_1} = \Cl(g)}
\mathop{\sum_{p_2\leq z}}_{\Frob_{p_2} = \Cl(g)}
\frac{\alpha_{\Cl(g)}}{(p_1-1) (p_2 -1)}\\
&+ O_A(N (\log N)^{-A}) =  \pi(N) (R + O(1))^2 .
 \end{aligned}\]
Hence, we conclude from (\ref{eq:kilkeat}) that
\[\begin{aligned} V &= R^2 \pi(N) - 2 R (R + O(1)) \pi(N) + \pi(N) (R+O(1))^2 +
\pi(N) (R+ O(1))\\ &= O(R \pi(N)).
\end{aligned}\]
Now, if 
\begin{equation}\label{eq:deltar}
\left|\mathop{\mathop{\sum_{p\leq z}}_{\Frob_p = \Cl(g)}}_{p|f(q)}
1 - R\;
\right| > \delta R
\end{equation}
for some $q\leq N$, $\delta>0$, then that value makes a contribution
greater
than $\delta^2 R^2$ to (\ref{eq:surt}). Hence there are at most
$\frac{O(R) \pi(N)}{\delta^2 R^2} = O\left(\frac{1}{\delta^2 R} \pi(N)\right)$
primes $q\leq N$ for which (\ref{eq:deltar}) is the case.
By the Chebotarev density theorem, 
\[\begin{aligned}
R &= (1 + o_f(1)) |\Cl(g)| \alpha_{\Cl(g)}\log \log z.
\end{aligned}\]
Hence $R\to \infty$ as $N\to \infty$, and so the statement of the lemma
follows.
\end{proof}
We will need a large-sieve lemma of a rather standard kind.
\begin{lem}\label{lem:pubha}
For any $N$ and any $\epsilon>0$,
\begin{equation}\label{eq:blond}
\sum_{r\leq N^{1/2-\epsilon}} \mathop{\sum_{\chi \mo r}}_{\text{$\chi$ primitive}}
\left|\mathop{\sum_{q\leq N}}_{\text{$q$ prime}} \chi(q)\right|^2 \ll_{\epsilon}
\frac{N^2}{(\log N)^2}.\end{equation}
\end{lem}
This is a special case of Problem 7.19 in \cite{IK}.
\begin{proof}
By the triangle inequality, the square root of the
left side of (\ref{eq:blond}) is at most
\[\sqrt{\sum_{r\leq N^{1/2-\epsilon}} \mathop{\sum_{\chi \mo r}}_{\text{$\chi$ primitive}}
\left|\mathop{\sum_{q\leq \sqrt{N}}}_{\text{$q$ prime}} \chi(q)\right|^2} \]
(which is $\ll \sqrt{N^{2 (1/2 - \epsilon)} (\sqrt{N}/\log N)^2} \ll N/\log N$)
plus the square-root of 
\begin{equation}\label{eq:gorty}
\sum_{r\leq N^{1/2-\epsilon}} \mathop{\sum_{\chi \mo r}}_{\text{$\chi$ primitive}}
\left|\mathop{\sum_{\sqrt{N} < q\leq N}}_{\text{$q$ prime}} \chi(q)\right|^2  .
\end{equation}
By \cite[Thm 8]{Bo} with $Q = \sqrt{N}$, (\ref{eq:gorty}) is at most
\[\frac{1}{\log \frac{\sqrt{N}}{N^{1/2-\epsilon}}} \cdot
(N + Q^2) \mathop{\sum_{\sqrt{N}< q \leq N}}_{\text{$q$ prime}} 1 \;\;
\ll_{\epsilon} \frac{N^2}{(\log N)^2}.\]
\end{proof}

\begin{lem}\label{lem:change}
Let $f\in \mathbb{Z}\lbrack x\rbrack$ be a polynomial irreducible over 
$\mathbb{Q}\lbrack x\rbrack$. For every prime $p$, write $p+1-a_p$ for
the number of points in $\mathbb{P}^2(\mathbb{Z}/p\mathbb{Z})$ 
on the curve $y^2 = f(x)$.
Let $z=z(N)$ be such that $z<N^{1/4-\epsilon}$, $\epsilon>0$, and
\[\lim_{N\to \infty} \sum_{p\leq z} \frac{a_p^2}{p^2} = \infty .\]
 Then, for a proportion $1+o(1)$ of all primes
$q\leq N$ as $N\to \infty$,
\begin{equation}\label{eq:strog}
\sum_{p\leq z} \frac{a_p}{p} \left(\frac{f(q)}{p}\right) =
(1 + o(1)) \sum_{p\leq z} \frac{-a_p^2}{p^2} ,\end{equation}
where the implied constants depend only on $\epsilon$.  
\end{lem}
Again, the proof will proceed by a variance bound.
\begin{proof}
Define 
\begin{equation}\label{eq:obstin}
V= \sum_{\text{$q\leq N$}} \left(\sum_{p\leq z} \frac{a_p}{p}
\left(\left(\frac{f(q)}{p}\right) + \frac{a_p}{p}\right)\right)^2,
\end{equation}
where, as per our convention, $q$ ranges only over the primes. Changing the
order of summation, we obtain
\begin{equation}\label{eq:bdjour}
V = \sum_{p_1\leq z} 
\frac{a_{p_1}}{p_1} \sum_{p_2\leq z} \frac{a_{p_2}}{p_2}
\sum_{q\leq N} \left(\left(\frac{f(q)}{p_1} \right) + \frac{a_{p_1}}{p_1}\right)
\left(\left(\frac{f(q)}{p_2} \right) 
+ \frac{a_{p_2}}{p_2}\right) .
\end{equation}
Expanding, we see that
\begin{equation}\label{eq:acquis}\begin{aligned}
V &= (R^2 + O(R))  \pi(N) + 2 R \sum_{p\leq z} 
\frac{a_{p}}{p} 
\sum_{a\mo p} 
\left(\frac{f(a)}{p} \right)
 |\{q\leq N : q\equiv a \mo p\}|\\
&+ \mathop{\sum_{p_1\leq z}  \sum_{p_2\leq z}}_{p_1\ne p_2} 
\frac{a_{p_1}}{p_1} \frac{a_{p_2}}{p_2}
\sum_{a \mo p_1 p_2} \left(\frac{f(a)}{p_1 p_2}\right) 
|\{q\leq N : q\equiv a \mo p_1 p_2\}|
\end{aligned}\end{equation}
where $R = \sum_{p\leq z} \frac{a_{p}^2}{p^2}$ and 
$\pi(N)$ denotes the number of primes $\leq N$.
  (The term $O(R)\cdot \pi(N)$
comes from the diagonal terms $p_1=p_2$
left out of the triple sum on the second line.)

We wish to approximate $|\{q\leq N : q\equiv a \mo p\}|$ 
(and $|\{q\leq N : q\equiv a \mo p_1 p_2\}|$) by
$\pi(N)/\phi(p) = \pi(N)/(p-1)$ for $a$ coprime to $p$ (or, respectively, by
$\pi(N)/\phi(p_1 p_2)$ for $a$ coprime to $p_1 p_2$). Now the
absolute value of
\[\sum_{p\leq z} \frac{a_{p}}{p} \left( 
\mathop{\sum_{a\mo p}}_{p\nmid a}  
\left(\frac{f(a)}{p} \right) \left|
 |\{q\leq N : q\equiv a \mo p\}| - \frac{\pi(N)}{p-1}\right| +
\left(\frac{f(0)}{p} \right)  |\{q\leq N : q\equiv 0 \mo p\}| \right)\]
is at most \[
\sum_{p\leq z} \left| \frac{a_{p}}{p} \right|
\mathop{\sum_{a\mo p}}_{p\nmid a}  
\left|
 |\{q\leq N : q\equiv a \mo p\}| - \frac{\pi(N)}{p-1}\right| +
\sum_{p\leq z} \left| \frac{a_p}{p} \right|.\]
By the trivial bound $|a_p|\ll p$, the second sum is $O(z)$ (and thus will
be negligible). 
We apply Cauchy-Schwarz twice to obtain that the first sum is at most
\begin{equation}\label{eq:hohei}
\sqrt{\sum_{p\leq z} \frac{a_p^2}{p^2}}
\sqrt{\sum_{p\leq z} (p-1) \mathop{\sum_{a\mo p}}_{p\nmid a}
\left| |\{q\leq N: q\equiv a \mo p\}| -
  \frac{\pi(N)}{p-1}\right|^2}.\end{equation}
The expression under the first square root is now $R$, which is 
$\ll \log z \ll \log N$. By a brief calculation, the expression under
the second square root equals
 \begin{equation}\label{eq:ester}
\sum_{p\leq z}\; \mathop{\sum_{\chi \mo p}}_{\text{$\chi$ non-principal}} 
|S(\chi)|^2\end{equation}
for $S(\chi) = \sum_{q\leq N} \chi(q)$,
where 
$q$ runs over the primes, as usual. By Lemma \ref{lem:pubha} (with
$\epsilon = 1/2$), (\ref{eq:ester}) is $O(\pi(N)^2)$. Hence 
(\ref{eq:hohei}) is at most $O(\sqrt{R} \pi(N))$. Therefore
\[\sum_{p\leq z} \frac{a_{p}}{p} \sum_{a\mo p} 
\left(\frac{f(a)}{p} \right)  |\{q\leq N : q\equiv a \mo p\}|
= \sum_{p\leq z} \frac{a_{p}}{p} \mathop{\sum_{a\mo p}}_{p\nmid a} 
\left(\frac{f(a)}{p} \right)  \frac{\pi(N)}{p-1}
+ O_A(\sqrt{R} \pi(N)) .\]
Now
\begin{equation}\label{eq:herring}\begin{aligned}
\mathop{\sum_{a\mo p}}_{p\nmid a} \left(\frac{f(a)}{p} \right) &=
 \sum_{a\mo p} \left(\frac{f(a)}{p} \right) -
 \left(\frac{f(0)}{p} \right) = 
 \sum_{a\mo p} |\{y\in \mathbb{Z}/p\mathbb{Z} : y^2 = f(a)\}| - p -
 \left(\frac{f(0)}{p} \right)\\
&= (p+1-a_p) + O(1) - p - \left(\frac{f(0)}{p} \right) = - a_p 
+ O(1),\end{aligned}\end{equation}
where the implied constant is absolute.
Thus
\[\begin{aligned}
\sum_{p\leq z} \frac{a_{p}}{p} \mathop{\sum_{a\mo p}}_{p\nmid a} 
\left(\frac{f(a)}{p} \right) \cdot  \frac{\pi(N)}{p-1} &=
\pi(N) \cdot \sum_{p\leq z} \frac{a_p}{p} \frac{1}{p-1} 
\left(- a_p + O(1)\right)\\
&= \pi(N) \left(\sum_{p\leq z} \frac{-a_p^2}{p^2} + O\left(
\sum_{p\leq z} \left( \frac{a_p^2}{p^3} + \frac{a_p}{p^2}\right) 
\right)
\right)\\ & = \pi(N) ( -R + O(1)),
\end{aligned}\]
where we use the Weil bound $|a_p|\ll \sqrt{p}$ in the last step.

Let us now estimate the sum in the second line of (\ref{eq:acquis}).
Since the only primes $q$ not coprime to $p_1$ or $p_2$ are $q=p_1$ and
$q=p_2$, the contribution of the terms with $\gcd(a,p_1 p_2)\ne 1$
is at most 
\[2 \sum_{p_1\leq z} \sum_{p_2\leq z} \frac{a_{p_1}}{p_1} \frac{a_{p_2}}{p_2} 
\ll z,\]
which is negligible. We write
\begin{equation}\label{eq:ustro}
\begin{aligned}&\mathop{\sum_{p_1\leq z}  \sum_{p_2\leq z}}_{p_1\ne p_2} 
\frac{a_{p_1}}{p_1} \frac{a_{p_2}}{p_2}
\mathop{\sum_{a \mo p_1 p_2}}_{\gcd(a,p_1 p_2)=1} 
\left(\frac{f(a)}{p_1 p_2}\right) 
|\{q\leq N : q\equiv a \mo p_1 p_2\}|
\\ &=
\mathop{\sum_{p_1\leq z}  \sum_{p_2\leq z}}_{p_1\ne p_2}
\frac{a_{p_1}}{p_1} \frac{a_{p_2}}{p_2} 
\mathop{\sum_{a \mo p_1 p_2}}_{\gcd(a,p_1 p_2) = 1}
 \left(\frac{f(a)}{p_1 p_2}\right) \frac{\pi(N)}{\phi(p_1 p_2)}
+
O\left(\mathop{\sum_{p_1\leq z}  \sum_{p_2\leq z}}_{p_1\ne p_2} 
\frac{|a_{p_1}|}{p_1} \frac{|a_{p_2}|}{p_2} \mathop{\sum_{a \mo p_1 p_2}}_{\gcd(a,
p_1 p_2)=1} \Delta_{a,p_1 p_2}\right)\\
&+
\mathop{\sum_{p_1\leq z}  \sum_{p_2\leq z}}_{p_1\ne p_2}
\frac{a_{p_1}}{p_1} \frac{a_{p_2}}{p_2} 
\mathop{\sum_{a \mo p_1 p_2}}_{\gcd(a,p_1 p_2) = 1}
 \left(\frac{f(a)}{p_1 p_2}\right)
\left(\frac{1}{\phi(p_1)}|\{q\leq N: q\equiv a \mo p_2\}| -
\frac{\pi(N)}{\phi(p_1 p_2)}\right)\\
&+
\mathop{\sum_{p_1\leq z}  \sum_{p_2\leq z}}_{p_1\ne p_2}
\frac{a_{p_1}}{p_1} \frac{a_{p_2}}{p_2}  
\mathop{\sum_{a \mo p_1 p_2}}_{\gcd(a,p_1 p_2) = 1}
 \left(\frac{f(a)}{p_1 p_2}\right)
\left(\frac{1}{\phi(p_2)}|\{q\leq N: q\equiv a \mo p_1\}| -
\frac{\pi(N)}{\phi(p_1 p_2)}\right),\end{aligned}\end{equation}
where
\[\begin{aligned}
\Delta_{a, p_1 p_2} &= 
|\{q\leq N : q\equiv a \mo p_1 p_2\}| -
\frac{1}{\phi(p_1)} |\{q\leq N: q\equiv a \mo p_2\}| \\ &-
\frac{1}{\phi(p_2)} |\{q\leq N: q\equiv a \mo p_1\}| +
\frac{1}{\phi(p_1 p_2)} \pi(N)
.\end{aligned}\]

The first sum on the right side of (\ref{eq:ustro}) is the main term;
by (\ref{eq:herring}), it equals
\[\pi(N)\cdot \mathop{\sum_{p_1\leq z}  \sum_{p_2\leq z}}_{p_1\ne p_2}
\frac{a_{p_1}}{p_1} \frac{a_{p_2}}{p_2}  
\frac{(-a_{p_1} + O(1)) (-a_{p_2} + O(1))}{\phi(p_1 p_2)} = \pi(N)
(R^2 + O(R)).\]

By Cauchy-Schwarz, the second sum 
in (\ref{eq:ustro}) (the sum within $O(\dotsc)$) 
is at most
\begin{equation}\label{eq:corto}
\sqrt{\mathop{\sum_{p_1\leq z}  \sum_{p_2\leq z}}_{p_1\ne p_2} 
\frac{a_{p_1}^2}{p_1^2} \frac{a_{p_2}^2}{p_2^2}}
\sqrt{\mathop{\sum_{p_1\leq z}  \sum_{p_2\leq z}}_{p_1\ne p_2} 
\left|\mathop{\sum_{a \mo p_1 p_2}}_{\gcd(a,p_1 p_2)=1} \Delta_{a,p_1 p_2}\right|^2}.
\end{equation}
The expression under the first square root is $\leq R^2$. 
By another application of Cauchy-Schwarz and a brief calculation
(cf. \cite[\S 2, Thm. 5]{Bo}), 
\[\begin{aligned}
&\left|\mathop{\sum_{a \mo p_1 p_2}}_{\gcd(a,p_1 p_2)=1} \Delta_{a,p_1 p_2}\right|^2
\leq \phi(p_1 p_2) 
\mathop{\sum_{a \mo p_1 p_2}}_{\gcd(a,p_1 p_2)=1} \left|\Delta_{a,p_1 p_2}\right|^2\\
&= \phi(p_1 p_2) 
\mathop{\sum_{a \mo p_1 p_2}}_{\gcd(a,p_1 p_2)=1} |\{q\leq N: q\equiv a \mo p_1 p_2
\}|^2 - \phi(p_1) 
\mathop{\sum_{a \mo p_1}}_{\gcd(a,p_1)=1} |\{q\leq N: q\equiv a \mo p_1 
\} |^2\\ &- \phi(p_2)
\mathop{\sum_{a \mo p_2}}_{\gcd(a,p_2)=1} |\{q\leq N: q\equiv a \mo p_2 
\}|^2 + \pi(N)^2 = \mathop{\sum_{\chi \mo p_1 p_2}}_{\text{$\chi$ primitive}} |S(\chi)|^2
.\end{aligned}\]
We apply Lemma \ref{lem:pubha}, and obtain that (\ref{eq:corto}) is 
$\ll_\epsilon \sqrt{R^2} \cdot \sqrt{\pi(N)^2} = R \pi(N)$.

By (\ref{eq:herring}),
the next to last line of (\ref{eq:ustro}) is
\[\begin{aligned}
&\sum_{p_1\leq z} \frac{a_{p_1}}{p_1} \cdot \frac{-a_{p_1} + O(1)}{p_1 - 1}
\mathop{\sum_{p_2\leq z}}_{p_2\ne p_1} \mathop{\sum_{a \mo p_2}}_{p_2\nmid a}
\left(|\{q\leq N : q\equiv a \mo p_2\}| - \frac{\pi(N)}{p_2 - 1}\right)\\
&\leq \left(- \sum_{p_1\leq z} \frac{a_{p_1}^2}{p^2} + O(1)\right)
\cdot 
\sum_{p_2\leq z} 
\mathop{\sum_{a \mo p_2}}_{p_2\nmid a} \left|
 |\{q\leq N : q\equiv a \mo p_2\}| - \frac{\pi(N)}{p_2 - 1}\right|.\end{aligned}\]
The first factor is $-R+O(1)$, whereas the second factor was already shown
before to be $O(\sqrt{R} \pi(N))$. Hence the next to last line
of (\ref{eq:ustro}) is
$O(R^{3/2} \pi(N))$. Obviously the same is true of the last line of 
(\ref{eq:ustro}).

Putting everything together, we see that  (\ref{eq:acquis}) has become
\[V = (R^2 + R) \pi(N) + 2 R (-R \pi(N) + O(\sqrt{R} \pi(N))) + 
(R^2 + O_\epsilon(R^{3/2})) \pi(N) = O_\epsilon(R^{3/2} \pi(N)).\]

Now, if 
\begin{equation}\label{eq:parano}
\left| \left(\sum_{p\leq z} \frac{a_p}{p} \left(\frac{f(q)}{p}\right)\right)
- (-R)\right| > \delta R
\end{equation}
for some $q\leq N$, $\delta>0$, then that value of $q$ makes a contribution 
greater than $\delta^2 R^2$ to $V$ (see (\ref{eq:obstin})).

Hence there are at most 
\[\ll \frac{R^{3/2} \pi(N)}{\delta^2 R^2} = \frac{\pi(N)}{
\delta^2 \sqrt{R}}\] primes $q\leq N$ for which (\ref{eq:parano}) is the case.
As $\lim_{N\to \infty} R = \infty$, we see that
$\pi(N)/(\delta^2 \sqrt{R}) = o_{\delta,\epsilon}(\pi(N))$ for any $\delta>0$.
 Since $\delta$ is arbitrarily small, the statement of the lemma
follows.
\end{proof}

\section{Rarity of typical twists: large deviations and higher moments}

We have seen (Lemmas \ref{lem:plusca} and \ref{lem:change}, plus
(\ref{eq:katmid})) that, if $q$ is a prime $\leq N$ lying outside a set
containing a proportion $o(1)$ of all primes $\leq N$, and $d y^2 = f(q)$,
where $y$ is a prime, then $d$ has some special properties: 
\begin{enumerate}
\item\label{it:sut} $\omega_{\Cl(g)}(d)$ must be of roughly a given size for each $g\in
  \Gal_f$, and
\item\label{it:sryu} \begin{equation}\label{eq:adalgisa}
\sum_{p\leq z} \frac{a_p}{p} \left( \frac{d}{p}\right) \sim - \sum_{p\leq
  z}
\frac{a_p^2}{p^2},\end{equation}
\end{enumerate}
i.e., $d$ will have a slight tendency to be a quadratic residue $\mo p$ when $a_p$ is
negative, and a nonresidue when $a_p$ is positive.

We will see in this section that only a small minority of all integers
$d\ll N (\log N)^{2 \epsilon}$ satisfy these properties. Here ``small
minority'' actually means ``fewer than $O((\log N)^{-(1+\delta)})$'', where
$\delta>0$ is fixed. This will be crucial later.

Let us first examine
how one would bound separately the number of integers satisfying
(\ref{it:sut})
and the number of integers satisfying (\ref{it:sryu}), i.e., equation
(\ref{eq:adalgisa}).
 (We will eventually
have to bound the number of integers satisfying both (\ref{it:sut}) and
(\ref{it:sryu}).)

One way of bounding $|\{d\ll N (\log N)^{2\epsilon}: \text{$d$ satisfies
  (\ref{it:sut})}\}|$ is to translate large-deviation estimates from
probability theory.
This was the approach followed in \cite{Hbr}. Here we will follow what
would look like a more familiar approach to an analytic number theorist,
though its content is essentially the same: we will bound expressions of
the form
\begin{equation}\label{eq:socorr}
\sum_d e^{\sum_i \alpha_i \omega_{S_i}(d)},
\end{equation}
where $\alpha_i\in \mathbb{R}$ will be chosen at will,
$\omega_{S_i} = \{p\in S_i: p|d\}$ and $S_i$ is a set of primes (in our
case, all unramified
primes with $\Frob_p$ equal to a fixed element of the Galois group). The
bounds will be the same as those given by large-deviation theory -- in
particular, there will be relative entropies in the exponents. 

How should we bound 
$|\{d\ll N (\log N)^{2\epsilon} : \text{$d$ satisfies
  (\ref{eq:adalgisa})}\}|$? A variance bound would not be good enough for
our purposes. If we could truly handle reduction modulo distinct primes
as so many independent
random variables, we would use an exponential moment bound. As mutual
independence does not truly hold, we will use instead a high moment, i.e.,
we will
bound
\begin{equation}\label{eq:massa}
\sum_d \left(\sum_{p\leq z} \frac{a_p}{p} \left(\frac{d}{q}\right) \right)^{2 k}
\end{equation}
for $k$ large.

As we said, we would actually like to bound the number of integers $d\ll N
(\log N)^{2\epsilon}$ satisfying both (\ref{it:sut})
and (\ref{it:sryu}) (i.e., equation
\ref{eq:adalgisa}). Getting an estimate that combines information from
both sources is, as we shall see, a technically delicate task, to be achieved
by the {\em enveloping} use of a sieve.

\begin{center}
* * *
\end{center}

The following lemma will allow us to work with small primes only without much 
of a loss in our estimates.
\begin{lem}\label{lem:acri}
For any $A>0$, $\epsilon>0$ and every $N$, there is a $z = z(N,A,\epsilon)$
with $\log \log z > (1-\epsilon) \log \log N$ and $z < N^{\epsilon}$ 
 such that, for all but
$O_{A,\epsilon}(N (\log N)^{-A})$ integers $n$ between $1$ and $N$,
\begin{enumerate}
\item\label{it:oyop} $\prod_{p|n: p\leq z} p^{v_p(n)} < N^{\epsilon}$,
\item\label{it:egg} $\omega(n) - \sum_{p|n: p\leq z}  1< \epsilon \log \log z$.
\end{enumerate}
\end{lem}
\begin{proof}
Apply \cite[Lemma 5.2]{Hbr} with $f(x)=x$ and $\epsilon/2$ instead of 
$\epsilon$; let $z = N^{\delta(N)}$. Then $\log \log z = \log \log N - \log \log \delta(N) > (1-\epsilon/2) \log \log N$. Furthermore,
$z = N^{o_{A,\epsilon}(1/\log \log N)} 
< N^{\epsilon}$ if (as we may assume) $N$ is larger than a constant depending on 
$A$ and $\epsilon$. 

By conclusion (a) in \cite[Lemma 5.2]{Hbr}, $\prod_{p|n: p\leq z} p < 
N^{\epsilon/2}$. It is also the case that the largest square factor in $n$ is 
$\leq N^{\epsilon/2}$ for all but $O\left(N^{1-\epsilon/4}\right)$ integers between
$1$ and $N$. Part (\ref{it:oyop}) of the statement follows.
 Conclusion (b) in \cite[Lemma 5.2]{Hbr} implies
that $\omega(n) - \sum_{p|n: p\leq z} 1 < (\epsilon/2) \log \log N$;
since $\log \log z > (1 - \epsilon/2) \log \log N \geq (1/2) \log \log N$,
part (\ref{it:egg}) of the statement follows immediately.
\end{proof}

The next lemma is both elementary and of a very classical type.
\begin{lem}\label{lem:sosv}
Let $S$ be a set of primes; define $S_z = \{p\in S: p\leq z\}$.
Assume that $\sum_{p\in S_z} 1/p \leq 
\beta \log \log z + C$, $C$ a constant. Let $N_z$ denote the set of all
positive integers that are products of primes in $S_z$ alone.
Let $\eta\geq 1$. Then
\[\mathop{\sum_{n\in N_z}}_{\omega(n) \geq \eta \beta \log \log z} \frac{1}{n}
\ll_{C, \eta} (\log z)^{\beta \cdot (\eta - \eta \log \eta)}\]
\end{lem}
The Lemma would still be true for $\eta < 1$ positive, but the exponent on
the right would no longer be optimal.
\begin{proof}
Recall that $\sum_{n\geq 1} 1/n^2 = \pi^2/6$.  For any $\alpha>0$,
\[\begin{aligned}
\left(\frac{\pi^2}{6}\right)^{\alpha} 
\prod_{p\in S_z} \left(1 + \frac{1}{p}\right)^{\alpha} &\geq 
 \prod_{p\in S_z} \left(1 + \frac{1}{p} + \frac{1}{p^2} + \dotsc 
\right)^{\alpha} \geq \prod_{p\in S_z}  
\left(1 + \frac{\alpha}{p} + \frac{\alpha}{p^2} + \dotsc\right)\\
&= \sum_{n\in N_z} \frac{\alpha^{\omega(n)}}{n} .\end{aligned}\]
Hence
\[\begin{aligned}
\mathop{\sum_{n\in N_z}}_{\omega(n) \geq \eta \beta \log \log z} \frac{1}{n}
&\leq \frac{1}{\alpha^{\eta \beta \log \log z}} \sum_{n\in N_z} \frac{\alpha^{\omega(n)}}{n} 
\leq \frac{1}{\alpha^{\eta \beta \log \log z}} \left(\frac{\pi^2}{6}\right)^{\alpha} 
\prod_{p\in S_z} \left(1 + \frac{1}{p}\right)^{\alpha}\\
&\ll_{C,\alpha} \frac{1}{\alpha^{\eta \beta \log \log z}} 
e^{\alpha\cdot  \beta \log \log z} = (\log z)^{(\alpha -
\eta \log \alpha) \beta}
.\end{aligned}\] 
To minimise $\alpha - \eta \log \alpha$, we set $\alpha = \eta$. Then
$(\log z)^{(\alpha -
\eta \log \alpha) \beta} = (\log z)^{\beta (\eta - \eta \log \eta)}$.
\end{proof}

\begin{lem}\label{lem:wrot}
Let $S$, $S'$ be sets of primes with
\begin{enumerate}
\item\label{it:ag} $S\subset S'$,
\item\label{it:bg} $\sum_{p\in S: p\leq z} 1/p \leq \beta \log \log z + C$ for all $z>e$, 
where $C$ is a constant,
\item\label{it:cg} $\sum_{n\leq z: p|n\Rightarrow p\in S'} 1/n \geq C' (\log z)^{\beta'}$ for all $z>e$, where $C'$ is a constant. 
\end{enumerate}

Let $N$ be a positive integer. Let $\eta>1$.
Let $B$ be the set of all integers $n\leq N$ having at least $\eta \beta 
\log \log N$ divisors in $S$, but no divisors in $S'\setminus S$. Then, for
all $\epsilon>0$ and every $A>0$, there is a sequence of non-negative
 reals $\{b_n\}_{n\leq N}$ such that
\begin{enumerate}
\item\label{it:concla} $b_n\leq \tau_5(n)$ for every $n$,
\item\label{it:conclb} $|\{n\in B: b_n < 1\}| \ll_{A,\epsilon} N/(\log N)^A$,
\item\label{it:conclc} $\sum_{n\leq N} b_n \ll_{C,C',\eta} N/(\log N)^{
(1-\epsilon/4) (\beta' + (\beta-\epsilon/4) (\eta \log \eta - \eta))}$,
\item\label{it:concld} $\sum_{n\leq N: n\equiv a \mo m} b_n
= \frac{1}{\phi(m)} \sum_{n\leq N: \gcd(n,m)=1} b_n
+ O_{\epsilon}\left(N^{\epsilon}\right)$
for every $m\leq N^{1-\epsilon}$ and every $a$ coprime to $m$.
\end{enumerate}
\end{lem}
The sequence $b_n$ is a variant of what is sometimes called an {\em enveloping
sieve}; here, as per conclusion (\ref{it:conclb}), the sequence $b_n$ almost
``envelops'' (i.e., majorises the characteristic function of)  $B$, but not quite. 
\begin{proof}
Let $z$ be as in Lemma \ref{lem:acri} with $\epsilon/4$ instead of
$\epsilon$; in particular, $z<N^{\epsilon/4}$. 
Let $\lambda_d$, $d\leq N^{\epsilon/2}$, be the weights in Selberg's
sieve\footnote{Brun's (non-pure) sieve or the Iwaniec-Rosser sieve 
(as in \cite[\S 6]{FI} and \cite[\S 11]{FI}, respectively) would do 
just as well as Selberg's sieve in this context. In fact, it would do 
slightly better, in that the subscript in conclusion (\ref{it:concla}) would go down from $5$ to $3$.}
 when used to sieve out prime factors $p\leq N^{\epsilon/4}$ in $S$. 
(Here we are using $\lambda_d$ to denote the sequence of non-negative reals
$\lambda_d$ (where $\lambda_d=0$ for $d>N^{\epsilon/2}$)
obtained by the identity
$\sum_{d|m} \lambda_d = \left(\sum_{d|m} \rho_d\right)^2$; here $\rho_d$ 
is as in, say, \cite[(7.15)]{FI}. In particular, $\lambda_1 = 1$ and
$|\lambda_d|\leq 1$ for all $d$. Note some other texts use an opposite
convention, exchanging the roles of $\lambda_d$ and $\rho_d$.) 

Define
\begin{equation}\label{eq:debussy}
b_n = \mathop{\mathop{\sum_{m|n, m\in N_z(S)}}_{\omega(m)\geq (\eta- \epsilon/4) \beta \log \log z}}_{m\leq N^{\epsilon/4}}
 \sum_{d|n/m, d\in N_z(S')} \lambda_d,\end{equation}
where, for a set $P$ of primes,
 $N_z(P)$ is the set of all positive integers that are products of primes
in $\{p\in P: p\leq z\}$ alone,

Since $\lambda_d\leq \tau_3(d)$, conclusion (\ref{it:concla}) is immediate. 
Let $n$ be in $B$. Then 
\[b_n \geq \mathop{\mathop{\mathop{\sum_{m|n, m\in N_z(S)}}_{\omega(m)\geq (\eta \beta -\epsilon/4) \log \log z}}_{m\leq N^{\epsilon/4}}}_{p|n/m\Rightarrow p\notin S} 1,
\]
since the condition $p|n/m \Rightarrow p\notin S$ ensures (given that 
$n$ has no divisors in $S'\setminus S$, thanks to $n\in B$) that 
the inner sum in (\ref{eq:debussy}) has $\lambda_1$ (which equals $1$) 
as its only term. By $n\in B$ and the definition of $B$, 
$n$ has at least $\eta \beta \log \log z$ divisors in $S$.
Hence $b_n$ can be less than $1$ only if, for $m=\prod_{p|n, p\in S: p\leq z} p^{v_p(n)}$,  either (a) $m>N^{\epsilon/4}$, (b) $\omega(n) - \omega(m) >
(\epsilon/4) \log \log z$. By Lemma \ref{lem:acri}, at most 
$O_{A,\epsilon}(N (\log N)^{-A})$ satisfy either statement (where $A>0$ 
is arbitrary). Hence conclusion (\ref{it:conclb}) holds.

Now
\[\sum_{n\leq N} b_n = \mathop{\sum_{m\leq N^{\epsilon/4},\; m\in N_z(S)}}_{
     \omega(m)\geq (\eta \beta - \epsilon/4) \log \log z}
        \sum_{n \leq N/m} \sum_{d|n, d\in N_z(S')} \lambda_d.\]
By the main result on the Selberg sieve (see, e.g., \cite[Thm.\ 7.1]{FI},
with $a_n = 1$ for all $n\leq N/m$, $a_n=0$ for $n>N/m$),
\[\begin{aligned}      \sum_{n \leq N/m} \sum_{d|n, d\in N_z(S')} \lambda_d &= 
\left(\prod_{p\in S': p\leq z}
 \frac{1}{1-1/p} \right)^{-1} 
N/m
+ O\left(\sum_{d<N^{\epsilon/2}} \tau_3(d)\right) \\ &\leq
\left(\mathop{\sum_{d\leq N^{\epsilon/4}}}_{ d\in N_z(S')} 1/d\right)^{-1} N/m + 
O_{\epsilon}\left(N^{3 \epsilon/4}\right) .
\end{aligned}\]
By condition (\ref{it:cg}) and $z < N^{\epsilon/4}$, we know that
$\sum_{d\leq N^{\epsilon/4}, d\in N_z(S')} 1/d \gg_{C'} (\log z)^{\beta'}$. 
Thus
 \[\sum_{n\leq N} b_n \ll_{C'} \frac{N}{(\log z)^{\beta'}}
\mathop{\sum_{m\leq N^{\epsilon/4},\; m\in N_z(S)}}_{
     \omega(m)\geq (\eta \beta - \epsilon/4) \log \log z} \frac{1}{m}
+ O_{\epsilon}\left(N^{\epsilon}\right).\]
We now apply Lemma \ref{lem:sosv}, and conclude that
\[\sum_{n\leq N} b_n \ll_{C,C',\eta} \frac{N}{(\log z)^{\beta' - (\beta-\epsilon/4) (\eta - \eta \log \eta)}}
.\]
Lemma \ref{lem:acri}
assures us that $\log \log z > (1 - \epsilon/4) \log \log N$, and so
$\log z > (\log N)^{1-\epsilon/4}$. We thus obtain conclusion (\ref{it:conclc}).

Lastly, for every $r$ and every $a$ coprime to $r$,
\[\begin{aligned}\sum_{n\leq N: n\equiv a \mo r} b_n &= 
\mathop{\sum_{m\leq N^{\epsilon/4}, m\in N_z(S)}}_{
     \omega(m)\geq (\eta \beta - \epsilon/4) \log \log z}
 \mathop{\sum_{d\leq N^{\epsilon/2}}}_{d\in N_z(S')} \lambda_d
        \mathop{\sum_{n \leq N/md}}_{n\equiv a \mo r} 1\\ 
&= \mathop{\sum_{m\leq N^{\epsilon/4}, m\in N_z(S)}}_{
     \omega(m)\geq (\eta \beta - \epsilon/4) \log \log z}
 \mathop{\sum_{d\leq N^{\epsilon/2}}}_{d\in N_z(S')} \lambda_d
        \left(\frac{1}{\phi(r)} \mathop{\sum_{n \leq N/md}}_{\gcd(n,r)=1} 1 + 
O(1)\right)\\ 
&= \frac{1}{\phi(r)} \sum_{n\leq N: \gcd(n,r)=1} b_n +
O\left(\sum_{m\leq N^{\epsilon/4}} \sum_{d\leq N^{\epsilon/4}} \lambda_d\right)\\
&= \frac{1}{\phi(r)} \sum_{n\leq N: \gcd(n,r)=1} b_n + O_{\epsilon}\left(
N^{\epsilon}\right),
\end{aligned}\]
i.e., conclusion (\ref{it:concld}) holds.
\end{proof}

We begin by an easy application of Lemma \ref{lem:sosv} to the case already
treated in \cite{Hbr}. We do this both for contrast with a later application 
(the proof of Prop.\ \ref{prop:kost}, which uses the divergence of
$\sum_p a_p^2/p^2$ and where the sieve does play an enveloping role) and to
make the paper relatively self-contained.

\begin{lem}\label{lem:moreug}
Let $K/\mathbb{Q}$ be a cubic extension of $\mathbb{Q}$ with Galois group 
$\Alt(3)$. Let $S$ be the set of unramified primes that split completely
in $K/\mathbb{Q}$. Let $V$ be the set of integers $n\leq N$ such that (a)
$n$ has at least $(1+o(1)) \log \log N$ divisors in $S$, (b) $n$ is not
divisible by any unramified primes outside $S$. Then, for every
$\epsilon>0$,
\[|V|\ll_{K,\epsilon} \frac{N}{(\log N)^{(1-\epsilon) \log 3}}.\]
\end{lem}
\begin{proof}
Let $S'$ be the set of all unramified primes. Note that
 conditions (\ref{it:ag}) and
(\ref{it:cg}) in Lemma \ref{lem:wrot} are clear, and condition (\ref{it:bg})
holds by the Chebotarev density theorem and partial summation. By
conclusions (\ref{it:conclb}) and (\ref{it:conclc}) in that lemma, applied
with $A=2$,
\[|V|\leq O_{\epsilon}(N/(\log N)^A) + \sum_{n\leq N} b_n \ll_{\epsilon}
\frac{N}{(\log N)^{(1-\epsilon) (1 + (1/3) (3 \log 3 - 3))}} = 
\frac{N}{(\log N)^{(1-\epsilon) \log 3}} .\] 
\end{proof}

The following is the more difficult case.
\begin{prop}\label{prop:kost}
Let $K/\mathbb{Q}$ be a cubic extension of $\mathbb{Q}$ with Galois group
$\Sym(3)$. Let $S$ be the set of unramified primes that split completely
in $K/\mathbb{Q}$;
let $S'$ be the set of unramified primes that either split completely
or are inert in $K/\mathbb{Q}$.
For every prime $p$, let $a_p$ be such that $|a_p|\leq 2 \sqrt{p}$ and, for $z = e^{(\log N)/(2 \log \log N)}$,
\begin{equation}\label{eq:impers}\left|\sum_{p\leq z} a_p^2/p^2 \right| = (1+o(1)) \log \log z.\end{equation} 

Let $V$ be the set of integers $n\leq N$ such that (a) $n$ has 
at least $(1/2+o(1)) \log \log N$ divisors in $S$, (b) $n$
has no divisors in $S'\setminus S$, (c) $n$ satisfies
\begin{equation}\label{eq:novar}\left|\sum_{p\leq z} \frac{a_p}{p} \left(\frac{n}{p}\right)\right| \geq (1 + o(1)) \log \log z\end{equation}
for $z$ as above. Then, for every $\epsilon>0$,
\[|V|\ll_{K,\epsilon} \frac{N}{(\log N)^{((1+\log 3)/2) - \epsilon}},\]
where the implied constant depends on $K$, $\epsilon$ and the implied constants in (a), (b), (\ref{eq:impers}) and (\ref{eq:novar}).
\end{prop}
\begin{proof}
We first verify that $S$ and $S'$  satisfy conditions (\ref{it:ag})-(\ref{it:cg}) of Lemma \ref{lem:wrot}. Condition (\ref{it:ag}) is obvious.
Condition (\ref{it:bg}) holds with $\beta = 1/6$ by the Chebotarev density
theorem.  Condition (\ref{it:cg}) holds for related reasons: as in
(say) the proof of \cite[Lemma 4.10]{Hsq}, we can write
\[\begin{aligned}
\prod_{p\in S'} \frac{1}{1-p^{-s}} &= 
\prod_p \left(\frac{1}{1-p^{-s}}\right) \cdot 
\left(\prod_{p\notin S'} \left(\frac{1}{1-p^{-s}}\right) 
\prod_{p\in S'\setminus S} \left(\frac{1}{1-p^{-s}}\right)^3\right)^{-1} \\ 
&\cdot
\left(\prod_{p\in S'\setminus S} \left(\frac{1}{1-p^{-s}}\right)^6\right)^{1/2} = L_1(s) \zeta(s)
\zeta_{K/\mathbb{Q}}(s)^{-1} \zeta_{L/\mathbb{Q}}(s)^{1/2} ,
\end{aligned}\]
where $L_1(s)$ is holomorphic and bounded on $\{s: \Re(s)>1/2+\epsilon\}$
and $L$ is the Galois closure of $K$. Since $\zeta$, $\zeta_{K/\mathbb{Q}}$
and  $\zeta_{L/\mathbb{Q}}$ each have a pole of order $1$ at $s=1$, we obtain 
\[\mathop{\sum_{n\leq z}}_{p|n\Rightarrow p\in S'} \frac{1}{n} \sim C (\log
z)^{1-1+1/2} = C (\log z)^{1/2}\]
for some constant $C$
by contour integration or a real Tauberian theorem
(e.g., a Hardy-Littlewood Tauberian theorem, \cite[Thm. 5.11]{MV}; there
is no need for a complex Tauberian theorem here).

Apply Lemma \ref{lem:wrot}. By conclusion (\ref{it:conclb}), we will find it enough to bound
$\sum_{n\in V} b_n$ from above: $|V|$ will exceed this sum by at most $O_A(N/(\log N)^A)$, where
we can set $A$ as large as needed. For any $k$, (\ref{eq:novar}) ensures that
\begin{equation}\label{eq:valls}\begin{aligned}
\sum_{n\in V} b_n &\leq \left(\max_{n\in V} \sum_{p\leq z} \frac{a_p}{p} \left(\frac{n}{p}\right)\right)^{-2k}
 \sum_{n\in V} b_n \left(\sum_{p\leq z} \frac{a_p}{p} \left(\frac{n}{p}\right)\right)^{2k} \\
&\leq \frac{1}{((1+o(1)) \log \log z)^{2k}} \sum_{n\leq N} b_n \left(\sum_{p\leq z} \frac{a_p}{p} \left(\frac{n}{p}\right)\right)^{2k} .
\end{aligned}\end{equation}

 The following amounts to a proof of a special case of Khintchin's inequality, generalised to the case of random
variables that are only approximately independent. First, we have
\begin{equation}\label{eq:gosi}
\sum_{n\leq N} b_n \left(\sum_{p\leq z} \frac{a_p}{p} \left(\frac{n}{p}\right)\right)^{2k} = \sum_{p_1,\dotsc,p_{2k}\leq z} \frac{a_{p_1}}{p_1} \dotsb \frac{a_{p_{2k}}}{p_{2k}} \sum_{n\leq N} b_n \left(\frac{n}{p_1} \right) \dotsb \left(\frac{n}{p_{2k}}\right) .
\end{equation}
Set $m = p_1 p_2 \dotsb p_{2k}$ and assume $m\leq N$. Using conclusions (\ref{it:concla}) and (\ref{it:concld})  in Lemma \ref{lem:wrot}, we get
\[\begin{aligned}
\sum_{n\leq N} b_n \left(\frac{n}{p_1}\right) \dotsc \left(\frac{n}{p_{2k}}\right) &=  
\mathop{\sum_{a \mo m}}_{\gcd(a,m)=1} \left(\frac{a}{p_1}\right) \dotsc \left(\frac{a}{p_{2k}}\right) 
\mathop{\sum_{n\leq N}}_{n\equiv a \mo m} b_n 
+ O\left(\sum_{n\leq m} b_n\right)\\
&= \mathop{\sum_{a \mo m}}_{\gcd(a,m)=1}  \left(\frac{a}{p_1}\right) \dotsc \left(\frac{a}{p_{2k}}\right)  \frac{1}{\phi(m)} 
\mathop{\sum_{n\leq N}}_{\gcd(n,m)=1} b_n \\ &+ O_{\epsilon}(N^{\epsilon}) \cdot \sum_{a \mo m} 1 + O\left(\sum_{n\leq m} \tau_5(m)\right)\\
&= \mathop{\sum_{n\leq N}}_{\gcd(n,m)=1} b_n \cdot \sum_{a \mo m} \left(\frac{a}{p_1}\right) \dotsb \left(\frac{a}{p_{2k}}\right) \frac{1}{\phi(m)}
+ O_{\epsilon}\left(N^{\epsilon} m\right) ,
\end{aligned}\]
provided that $z^{2k} \leq N$. If there is a $p$ appearing an odd number of times in $p_1,p_2,\dotsc,p_{2k}$, the sum $\sum_{a\mo m}
\left(\frac{a}{p_1}\right) \dotsb \left(\frac{a}{p_{2k}}\right)$ vanishes. On the other hand, given a multiset $S$ consisting
of $k$ not necessarily distinct primes, 
the number of distinct tuples $(p_1, p_2,\dotsc,p_{2k})$ such that every prime $p$ appearing exactly $\ell$
times in $S$ appears exactly $2\ell$ times in $p_1,p_2,\dotsc,p_{2k}$ 
is at most $(2k)!/(2^k k!)$ times the number of tuples $(q_1,q_2,
\dotsc,q_k)$ such
that $S = \{q_1,q_2,\dotsc,q_k\}$. (This is so by the crude bound 
$(2r)! \geq 2\cdot r!$ for $r\geq 1$.)
Hence, going back to (\ref{eq:gosi}) and using conclusion (\ref{it:conclc}) in Lemma \ref{lem:wrot}, we obtain
\[\begin{aligned}
\sum_{n\leq N} b_n \left(\sum_{p\leq z} \frac{a_p}{p} \left(\frac{n}{p}\right)\right)^{2k} &\leq \frac{(2k)!}{2^k k!}
\sum_{q_1,\dotsc,q_k\leq z} \frac{a_{q_1}^2}{q_1^2} \dotsb \frac{a_{q_k}^2}{q_k^2} \sum_{n\leq N} b_n \\ 
&+ O_{\epsilon}\left(
\sum_{p_1,\dotsc,p_{2k}\leq z} \frac{a_{p_1}}{p_1} \dotsb \frac{a_{p_{2k}}}{p_{2k}} \cdot N^{\epsilon} p_1 p_2 \dotsb p_{2k}\right)\\
&\ll_{f,\epsilon} \frac{(2k)!}{2^k k!} \frac{N}{(\log N)^{(1-\epsilon/4) (1/2 + (1/6 - \epsilon/4) (3 \log 3 - 3))}} \left(\sum_{p\leq z} \frac{a_p^2}{p^2}\right)^k
\\&+O_{\epsilon}\left(\left(\sum_{p\leq z} \frac{2 \sqrt{p}}{p}\right)^k N^{\epsilon} z^{2k}\right) \\
&\leq \frac{(2k)!}{2^k k!} \frac{N}{(\log N)^{(1-\epsilon) (\log 3)/2}} 
(1+o(1))^k (\log \log z)^k + O_{\epsilon}\left(N^{\epsilon} z^{3k}\right) .
\end{aligned}\]

Thus, by (\ref{eq:valls}),
\[\begin{aligned}
\sum_{n\in V} b_n &\leq \frac{((1+o(1)) \log \log z)^{k}}{((1+o(1)) \log \log z)^{2k}} \frac{(2k)!}{2^k k!} \frac{N}{(\log N)^{(1-\epsilon) (\log 3)/2}} 
+ O_{\epsilon}\left(N^{\epsilon} z^{3k}\right)\\
&\ll \frac{e^{-k} (2k)^k}{((1+o(1)) \log \log N)^k} \frac{N}{(\log N)^{(1-\epsilon) (\log 3)/2}}  + O_{\epsilon}\left(N^{\epsilon} e^{\frac{3}{2}
k (\log N)/\log \log N}\right).
\end{aligned}\]
We set $k = (\log \log N)/2$, and obtain
\[\begin{aligned}
\sum_{n\in V} b_n  &\ll \frac{(\log N)^{-1/2} (2k)^k}{(1+o(1))^{\log \log N/2} (2k)^k}  \frac{N}{(\log N)^{(1-\epsilon) (\log 3)/2}}
+ O(N^{3/4+\epsilon})\\
&\ll_{\epsilon} \frac{N}{(\log N)^{(1+\log 3)/2 - \epsilon}} + O(N^{3/4+\epsilon}).
\end{aligned}\]
\end{proof}

\section{Modularity. Conclusion.}
It remains to estimate $\sum_{p\leq z} a_p^2/p^2$, where, as usual, we define 
$a_p$ by letting $p+1-a_p$ be the number of (projective)
points $\mod p$ on the curve $y^2 = f(x)$. Our estimate will be based on
 the fact
that the Rankin-Selberg $L$-function $L_{f\otimes f}$ has a pole
at $s=2$.

\begin{lem}\label{lem:rsel}
Let $f\in \mathbb{Z}\lbrack x\rbrack$ be a cubic polynomial irreducible over
$\mathbb{Q}\lbrack x\rbrack$. For every prime $p$, write $p+1-a_p$ for
the number of points in $\mathbb{P}^2(\mathbb{Z}/p\mathbb{Z})$ 
on the curve $y^2 = f(x)$.
Then, as $x\to \infty$,
\[\sum_{p\leq x} a_p^2/p^2 = (1+o_f(1)) \log \log x .\]
\end{lem}
\begin{proof}
By the modularity of elliptic curves (\cite{W}, \cite{TW}, \cite{BCDT}), there
is a primitive cusp form $f$ of weight $2$ and level $N$ such that 
$f(z) = \sum_{n=1}^{\infty} a_n n^{1/2} e(n z)$. The Rankin-Selberg $L$-function
 $L(f\otimes \overline{f},s) = \sum_{n=1}^\infty |a_n|^2 n^{-s-1} = 
\sum_{n=1}^{\infty} a_n^2 n^{-s-1} = L(f\otimes f,s)$
(\cite[(13.49)]{Iw}, where $a(n) = n^{-1/2} a_n$) 
then has a simple pole at $s=1$ 
(the residue given by \cite[(13.52)]{Iw} is non-zero). 
Its Euler product decomposition is
\[\begin{aligned}
L(f\otimes f,s) &= \prod_p (1 + p^{-s}) (1 - \alpha_p^2 p^{-s})^{-1} (1- p^{-s})^{-1}
(1 - \beta_p^2 p^{-s})^{-1}\\
&= \frac{1}{\zeta(2s)} 
\prod_p (1 - p^{-s})^{-2} (1 - \alpha_p^2 p^{-s})^{-1} (1 - \beta_p^2 p^{-s})^{-1}\\
,\end{aligned}\]
where $\alpha_p$, $\beta_p$ are the reals satisfying 
$\alpha_p + \beta_p = a_p/\sqrt{p}$
and $\alpha_p \beta_p = 1$. 

Now
\[\begin{aligned}
- \frac{L'(f\otimes f,s)}{L(f\otimes f,s)} &= (- \log L(f\otimes f,s))' \\ &= 
2 \frac{\zeta'(2s)}{\zeta(2s)} + \sum_{p} (\log p) \sum_{m\geq 1} p^{-ms}
(2 + \alpha_p^{2m} +\beta_p^{2m})\\
&= \sum_p (\log p) a_p^2 p^{-s} + G(s),\end{aligned}\]
where $G(s)$ is holomorphic for $\Re(s)>1/2$. 

Because $L(f\otimes f,s)$ has a simple pole at $s=1$,
the function $- L'(f\otimes f,s)/L(f\otimes f,s)$ has a simple pole with residue $1$ at $1$.
It is now enough to apply a Tauberian theorem of Hardy-Littlewood type
\cite[Thm. 5.11]{MV}; we obtain
\[\sum_{n\leq x} (\log p) a_p^2/p^2 \sim \log x,\]
which, by partial summation, gives
\[\sum_{n\leq x} \frac{a_p^2}{p^2} \sim \log \log x,\]
as desired.
\end{proof}

\begin{proof}[Proof of main theorem]
By (\ref{eq:ali}), it is enough to show that
\[|\{p\leq N:  \exists q  \;\text{s.t.}\; q^2|f(p),\; q\geq  N (\log N)^{-\epsilon}  \}| = o(N/\log N)\]
for some $\epsilon>0$ independent of $N$. (Recall $p$ and $q$ both denote 
primes.)
If $f$ is reducible,
 the problem reduces to that with $f$ replaced by each of 
its irreducible factors $g$ (since $p^2|f(n)$ for any prime $p$ not dividing the
discriminant $\Disc(f)$ 
implies $p^2|g(n)$ for some irreducible factor $g$ of $f$) and then,
since $\deg(g)\leq 2$, we have the problem solved by Estermann \cite{Es} 
(use simply \cite[Lemma 6.2]{Hbr}).  

We can thus assume that $f$ is an
irreducible polynomial. We can also assume without loss of generality
 that the leading coefficient of
$f$ is positive. Let $\alpha$ be a root of $f(x)=0$. Define 
$K = \mathbb{Q}(\alpha)/\mathbb{Q}$.

Let $N' = \max_{n\leq N} f(n)/(N(\log N)^{-\epsilon})^2$. Clearly
$N' \sim c_f N (\log N)^{2\epsilon}$, where $c_f$ is the leading coefficient of 
$f$.
Let $z = e^{(\log N')/(2 \log \log N')}$. Let $S$ be the set of unramified primes that
split completely in $K/\mathbb{Q}$. 
By Lemma \ref{lem:plusca}, the number of
primes in $S$ dividing $f(p)$ is $(3 + o_f(1)) (1/6) \log \log z =
(1/2 + o_f(1)) \log \log z$ (if $\Gal_{K/\mathbb{Q}} = \Sym(3)$) or
$(3+ o_f(1)) (1/3) \log \log z = (1 + o_f(1)) \log \log z$ 
(if $\Gal_{K/\mathbb{Q}} = \Alt(3)$) 
 for all but $o_f(N/\log N)$ primes $p\leq N$. (A prime $p$ that
splits completely has $\Frob_p$ equal to $\{e\}$, where $e$ is the identity
element of the Galois group.)
The number of primes in $S$ dividing $f(p)/q^2$ differs from this by
at most $1$, and thus is also $(1/2 + o_f(1)) \log \log z$
(if $\Gal_{K/\mathbb{Q}} = \Sym(3)$) or $(1+ o_f(1)) \log \log z$
(if $\Gal_{K/\mathbb{Q}} = \Alt(3)$).
Note that no unramified prime inert in $K/\mathbb{Q}$ can divide
$f(p)$ (and thus no such prime can divide $f(p)/q^2$).

Suppose first that $\Gal_{K/\mathbb{Q}} = \Alt(3)$.
Lemma \ref{lem:moreug} (applied with $N'$ instead of $N$)
gives us that there are at most
\[O_{f,\epsilon}(N/(\log N)^{\log 3 - 4 \epsilon})\] 
possible values of $d = f(p)/q^2$,
where $p$ ranges across the primes $p\leq N$, with $o_f(N/\log N)$ primes
excluded. Let $D$ be the set of such values $d$.

Suppose now that $\Gal_{K/\mathbb{Q}} = \Sym(3)$.
By Lemma \ref{lem:rsel},
$\sum_{p\leq z} a_p^2/p^2 = (1 + o_f(1)) \log \log z$; we can thus apply 
Lemma \ref{lem:change}, and obtain that, for all but $o_f(N/\log N)$ primes
$p\leq N$,
\[\sum_{p'\leq z} \frac{a_{p'}}{p'} \left(\frac{f(p)/q^2}{p'}\right) = 
O(1) + \sum_{p'\leq z} \frac{a_{p'}}{p'} \left(\frac{f(p)}{p'}\right) = 
- (1 + o(1)) \log \log z.\]
Proposition \ref{prop:kost} (applied with $N'$ instead of $N$)
 now gives us that there are at most
\[O_{f,\epsilon}(N/(\log N)^{((1+ \log 3)/2)-3 \epsilon})\] possible values of $d = f(p)/q^2$,
where $p$ ranges across the primes $p\leq N$, with $o_f(N/\log N)$ primes
excluded. Let $D$ be the set of such values $d$.

 We now use Prop.\ \ref{prop:spire}, and obtain that the
numbers of integers (prime or not) $1\leq x\leq N$ 
such that $d q^2 = f(x)$ for some $d\in D$ and some
integer $q\geq N (\log N)^{-\epsilon}$ is
at most $O_{f,\epsilon}\left(N/(\log N)^{(\log 3) -4 \epsilon}\right)$
(if $\Gal_{K/\mathbb{Q}} = \Alt(3)$) 
or at most $O_{f,\epsilon}\left(N/(\log N)^{((1+ \log 3)/2)-4 \epsilon}\right)$
(if $\Gal_{K/\mathbb{Q}} = \Sym(3)$). Since $\log 3 > 1$ and
$(1 + \log 3)/2 > 1$, we are done. 
\end{proof}



\begin{thebibliography}{GT2}
\bibitem{BCDT} C. Breuil, B. Conrad, F. Diamond and R. Taylor, On the 
modularity of elliptic curves over $\mathbb{Q}$: wild $3$-adic exercises,
{\em J. Amer. Math. Soc.} {\bf 14} (2001), 843--939.
\bibitem{Bo}
Bombieri, E., {\em Le grande crible dans la th\'eorie analytique des nombres},
{\em Ast\'erisque} {\bf 18} (1987).
\bibitem{BFI} Bombieri, E., Friedlander, J. B., and H. Iwaniec,
Primes in arithmetic progressions to large moduli, {\em Acta Math.}
{\bf 156} (1986), no. 3--4, 203--251.
\bibitem{BK} Brumer, A., and K. Kramer, The rank of elliptic curves, {\em Duke Math. J.}
{\bf 44} (1977), 715--743.
\bibitem{Br} Browning, T., Power-free values of polynomials, {\em Arch. Math.}
{\bf 96} (2011), 139--150.
\bibitem{Er} Erd\H{o}s, P., Arithmetical properties of polynomials, {\em J.
London Math. Soc.} {\bf 28} (1953), 416--425.
\bibitem{Es}
Estermann, T., Einige S\"{a}tze \"uber quadratfreie Zahlen, {\em Math. Ann.}
{\bf 105} (1931), 653--662.
\bibitem{FI} Friedlander, J., and H. Iwaniec, {\em Opera de cribro}, AMS, Providence,
 2010.
\bibitem{Gr} Greaves, G., Power-free values of binary forms, {\em Quart. J.
Math. Oxford} {\bf 43} (2) (1992), 45--65.
\bibitem{HeBr} Heath-Brown, D. R., Counting rational points on algebraic varieties, {\em Analytic number theory}, 51--95, {\em Lecture notes in Math.}, v. 1891,
Springer-Verlag, 2006.
\bibitem{Hsq} Helfgott, H., On the square-free sieve, {\em Acta Arith.} {\bf 115} (2004), 349--402.
\bibitem{Hbr} Helfgott, H., Power-free values, large deviations and integer points on irrational
curves, {\em J. Th\'eor. Nombres Bordeaux} {\bf 19} (2007), 433--472.
\bibitem{Hsh} Helfgott, H., Power-free values, repulsion between points, differing beliefs and the existence of error, {\em CRM Proceedings and Lecture Notes}, v. 46 (2008), 81--88.
\bibitem{HV} Helfgott, H. A., and A. Venkatesh, Integral points on elliptic curves and $3$-torsion
in class groups, {\em J. of the Am. Math. Soc.} {\bf 19} (2006), 527--550.
\bibitem{Hoo}
Hooley, C., {\em Applications of sieve methods to the theory of numbers},
Cambridge University Press, Cambridge, 1976.
\bibitem{IK}
Iwaniec, H., and E. Kowalski, {\em Analytic Number Theory}, AMS Colloquium
publications, v. 53, AMS, Providence, RI, 2004.
\bibitem{Iw}
Iwaniec, H., {\em Topics in classical automorphic forms}, AMS, Providence, 1997.
\bibitem{KL}
Kabtjanski\u\i, G. A., and V. I. Leven\v ste\u\i n, Bounds for packings
on the sphere and in space, {\em Problemy Pereda\v ci Informacii}
{\bf 14} (1978), no. 1, 3--25.
\bibitem{MV}
H. Montgomery and R. Vaughan, {\em Multiplicative number theory:
  I. Classical Theory}, Cambridge University Press, 2007.
\bibitem{Na}
Nair, M., Power-free values of polynomials, II, {\em Proc. London Math. Soc. 
(3)} {\bf 38} (1979), no. 2, 353--368.
\bibitem{Sa} Salberger, P., Counting rational points on projective
  varieties, preprint, 2010.
\bibitem{Si} Silverman, J. H., A quantitative version of Siegel's theorem: integral points on 
elliptic curves and Catalan curves, {\em J. Reine Angew. Math.} {\bf 378}
(1987), 60--100.
\bibitem{TW}
Taylor, R., and A. Wiles, Ring-theoretic properties of certain Hecke algebras,
{\em Ann. of Math.} (2) {\bf 141} (1995), pp. 553--572.
\bibitem{Tu}
Tur\'an, P., On a Theorem of Hardy and Ramanujan, {\em J. London
  Math. Soc.} {\bf 9} (1934) 274--276. 
\bibitem{W}
Wiles, A., Modular elliptic curves and Fermat's last theorem, {\em Ann. of 
Math.} (2) {\bf 141} (1995), pp. 443--551.
\end{thebibliography}
\end{document}